\documentclass[11pt]{amsart}

\usepackage{epigamath}

\usepackage[english]{babel}

\usepackage[utf8]{inputenc}
\usepackage[T1]{fontenc}
\usepackage{dsfont}
\usepackage{xfrac}
\usepackage{tikz}
\usetikzlibrary{matrix,arrows,decorations.pathmorphing}
\usepackage{microtype}

\newtheorem{satz}{Satz}
\newtheorem{thm}[satz]{Theorem}
\newtheorem{lem}[satz]{Lemma}
\newtheorem{cor}[satz]{Corollary}
\newtheorem{pro}[satz]{Proposition}
\newtheorem{ex}[satz]{Example}
\newtheorem{defi}[satz]{Definition}
\newtheorem{rem}[satz]{Remark}

\renewcommand{\O}{\mathcal{O}}
\renewcommand{\P}{\mathds{P}}

\renewcommand{\tilde}{\widetilde}
\renewcommand{\H}{\textup{H}}

\newcommand*{\defeq}{\mathrel{\vcenter{\baselineskip0.5ex \lineskiplimit0pt
			\hbox{\scriptsize.}\hbox{\scriptsize.}}}%
	=}
\newcommand*{\eqdef}{=\mathrel{\vcenter{\baselineskip0.5ex \lineskiplimit0pt
			\hbox{\scriptsize.}\hbox{\scriptsize.}}}%
	}

\newcommand\restr[2]{
	{\left.\kern-\nulldelimiterspace#1\vphantom{\big|}\right|_{#2}}
}

\let\enumerateO\enumerate
\let\endenumerateO\endenumerate
\renewenvironment{enumerate}{\enumerateO\setlength{\parskip}{0em}\setlength{\parindent}{1em}}{\endenumerateO}

\makeatletter
\newcommand{\mylabel}[2]{(#2)\def\@currentlabel{#2}\label{#1}}
\makeatother

\newcommand{\Z}{\ensuremath{\mathds{Z}}}
\newcommand{\Q}{\ensuremath{\mathds{Q}}}
\newcommand{\R}{\ensuremath{\mathds{R}}}
\newcommand{\C}{\ensuremath{\mathds{C}}}
\newcommand{\N}{\ensuremath{\mathds{N}}}

\newcommand{\Np}{\ensuremath{\mathds{N}_{>0}}}

\newcommand{\fH}{\ensuremath{\mathcal{H}}}
\newcommand{\fI}{\ensuremath{\mathcal{I}}}

\newcommand{\fN}{\ensuremath{\mathcal{N}}}

\newcommand{\fT}{\ensuremath{\mathcal{T}}}

\newcommand{\fm}{\ensuremath{\mathfrak{m}}}
\newcommand{\sing}{\ensuremath{\mathrm{sing}}}
\newcommand{\reg}{\ensuremath{\mathrm{reg}}}

\DeclareMathOperator{\codim}{codim}

\DeclareMathOperator{\ra}{\rightarrow}
\DeclareMathOperator{\inj}{\hookrightarrow}
\DeclareMathOperator{\lra}{\longrightarrow}

\DeclareMathOperator{\Ra}{\Rightarrow}

\DeclareMathOperator{\tr}{tr}

\DeclareMathOperator{\h}{h}
\DeclareMathOperator{\Alb}{Alb}

\DeclareMathOperator{\Aut}{Aut}

\EpigaVolumeYear{4}{2020}
\EpigaArticleNr{7}
\ReceivedOn{June 1st, 2018}
\InFinalFormOn{March 19, 2020}
\AcceptedOn{April 18, 2020}

\title{Fujiki relations and fibrations of irreducible symplectic varieties}
\titlemark{Fujiki relations and fibrations of irreducible symplectic varieties}

\author{Martin Schwald}
\address{Martin Schwald, Universität Duisburg-Essen, Fakultät für Mathematik, Thea-Leymann-Str.\,9, 45127 Essen}
\email{\href{mailto:martin.schwald@uni-due.de}{martin.schwald@uni-due.de}}

\authormark{M. Schwald}

\AbstractInEnglish{This paper concerns different types of singular complex projective varieties generalizing irreducible symplectic manifolds. We deduce from known results that the generalized Beauville-Bogomolov form sa\-tis\-fies the Fujiki relations and has rank $\big(3,0,b_2(X)-3\big)$. This enables us to study fibrations of these varieties; imposing the newer definition from \cite[Definition~8.16.2]{GKP16} we show that they behave much like irreducible symplectic manifolds.}



\TitleInFrench{Relations de Fujiki et fibrations de variétés symplectiques irréductibles}

\AbstractInFrench{Cet article concerne différents types de variétés projectives complexes singulières qui généralisent les variétés symplectiques irréductibles. Nous déduisons de résultats connus que la forme quadratique de Beauville-Bogomolov généralisée vérifie les relations de Fujiki et a pour signature $(3,0,b_2(X)-3)$. Cela nous permet d'étudier les fibrations de ces variétés ; en imposant la récente définition~8.16.2 issue de \cite{GKP16}, nous montrons qu'elles se comportent de façon très similaire aux variétés symplectiques irréductibles lisses.}

\acknowledgement{The author gratefully acknowledges support by the DFG-Graduiertenkolleg GK1821
  ``Cohomological Methods in Geometry'' at the University of Freiburg and partial support by DFG-SFB/TRR 45.}

\begin{document}



\maketitle

\begin{prelims}

\DisplayAbstractInEnglish

\bigskip

\DisplayKeyWords

\medskip

\DisplayMSCclass

\bigskip

\languagesection{Fran\c{c}ais}

\bigskip

\DisplayTitleInFrench

\medskip

\DisplayAbstractInFrench

\end{prelims}


\newpage

\setcounter{tocdepth}{1} 

\tableofcontents


\section{Introduction}
  
  \subsection{Symplectic varieties} \emph{Irreducible symplectic manifolds} were introduced by Beauville \cite[p.~763f]{Bea83} and play an important role in the structure theory of complex manifolds of Kodaira dimension zero as they occur as factors in the famous Beauville-Bogomolov decomposition theorem \cite[th.~2]{Bea83}.
A singular version of the decomposition theorem was recently proved \cite[Theorem~1.5]{HP19}, building on the work of \cite{GKP16, GGK17, DG18, Dru18}.
Remarkably, the class of singular symplectic factors emerging from this approach \cite[Definition~8.16.2]{GKP16}, we call them here \emph{primitive symplectic varieties}, differs from the types of singular symplectic varieties studied before.
  
 To understand the geometry of \emph{primitive symplectic varieties} better, we want to classify their fibrations.
 Our main results show that they behave much like irreducible symplectic manifolds and are arguably closer to them than other definitions. While we prove in Theorem~\ref{mainfujiki} and \ref{fibration} that many of the properties that are well known for irreducible symplectic manifolds hold for a bigger class of symplectic varieties, which we call here \emph{irreducible symplectic varieties}, Theorem~\ref{fibration2} holds for primitive symplectic varieties but not completely for \emph{irreducible symplectic varieties}.
  
 We briefly recall the relevant definitions required to state our main results. The different notions in the literature generalizing holomorphic symplectic manifolds to a singular setting agree in the existence of a \emph{symplectic form} on the smooth locus, which is best seen as a \emph{reflexive} $2$-form.
  
  \emph{Reflexive differential forms} are holomorphic $p$-forms on the smooth locus $X_{\reg}$ of a normal variety $X$. We denote the associated sheaf on $X$ by $\Omega^{[p]}_X$ and get
\[
  \Omega^{[p]}_X\cong i_*\Omega^p_{X_{\reg}}\cong(\Omega^p_X)^{\ast\ast}\quad\text{and}\quad\H^0(X,\,\Omega^{[p]}_X)\cong\H^0(X_{\reg},\,\Omega_{X_{\reg}}^p),
\]
  where $i\colon X_{\reg}\inj X$ is the inclusion. These forms satisfy good pull back properties, see Section~\ref{extensionsection}. We recommend \cite{MR597077} for a reference on \emph{reflexive sheaves} and \cite[I--III]{GKKP11}, \cite[I]{Kebekus2016} for more information on reflexive differential forms.

A \emph{quasi-\'{e}tale} morphism $\pi\colon X'\ra X$ is a finite surjective morphism between complex varieties that is \'{e}tale outside of an analytic subset $Z\subset X'$ with $\codim Z\geq2$.
\begin{defi}[Symplectic varieties]\label{symplectic}
  Let $X$ be a normal, complex projective variety.
	\begin{enumerate}
	\item A \emph{symplectic form} on $X$ is a reflexive form $\omega \in\H^0(X,\, \Omega^{[2]}_X)$ that is non-degenerate at every point $x\in X_{\reg}$.
	\item The pair $(X,\omega)$ is called a \emph{symplectic variety} if for every re\-so\-lution of singularities $\nu\colon\tilde{X}\ra X$ there is a holomorphic form ${\tilde{\omega}\in\H^0(\tilde{X},\,\Omega^2_{\tilde{X}})}$ that coincides with $\nu^*\omega$ over $X_{\reg}$. We say that $\nu^*\omega$ \emph{extends} to $\tilde{\omega}$ on $\tilde{X}$.
  \item A symplectic variety $(X,\omega)$ is called \emph{irreducible symplectic} if its irregularity $\h^1(X,\,\O_X)$ vanishes and if
  $\h^0(X,\,\Omega^{[2]}_X)=1$.
	\item We call an irreducible symplectic variety $(X,\omega)$ \emph{Namikawa symplectic} if $X$ is $\Q$-factorial with\\ $\codim_{\C}X_{\sing} \geq 4$.
	\item	We call a symplectic variety \emph{primitive symplectic} if for all quasi-\'{e}tale morphisms $\pi\colon X'\ra X$ with $X'$ normal the reflexive\footnote{This reflexive pullback exists e.g. due to \cite[Proposition~1.6]{MR597077}} pullback $\pi^*\omega\in\H^0(X',\,\Omega^{[2]}_{X'})$ generates the exterior algebra of reflexive forms on $X'$, so
	\[
	\bigoplus_{p\in\N}\H^0(X',\,\Omega^{[p]}_{X'})=\C\left[\pi^*\omega\right].
	\]
	\end{enumerate}
\end{defi}

Section~\ref{propertiessection} reviews properties of these notions and discusses their relations. Symplectic varieties were introduced by Beauville \cite[Definition~1.1]{Bea00}.

Namikawa symplectic varieties, most prominently studied by Namikawa and Matsushita, have a well-behaved deformation theory \cite{Nam01}.
Primitive symplectic varieties are building blocks in a singular version of the Beauville-Bogomolov decomposition theorem.
Using \cite[Proposition~6.9]{GKP16}, we can consider irreducible symplectic varieties as a natural generalization of these two notions. 

\emph{Irreducible symplectic manifolds} are by definition simply connected, smooth irreducible symplectic varieties \cite[p.~763f]{Bea83}. Using the Beauville-Bogomolov decomposition theorem we show that a smooth irreducible symplectic variety is always either simply connected or an \'{e}tale quotient of an Abelian variety by a finite group of biholomorphic automorphisms, see Lemma~\ref{simplyconnected}. It is unknown if the latter case really occurs.

Irreducible symplectic manifolds are always primitive symplectic by \cite[prop.~3]{Bea83}. In the singular case being primitive symplectic is known to be more restrictive, see Example~\ref{ex}. This is our motivation to study how their geometric properties differ.

\subsection{Main results} On an irreducible symplectic variety $X$, Namikawa and Kirschner constructed in \cite[Theorem~8~(2)]{Nam01}, \cite[Definition~3.2.7]{Kir15} an important quadratic form $q_X$ on $\H^2(X,\,\C)$, the \emph{Beauville-Bogomolov form}. Its definition will be recalled in Section~\ref{generalizedform}. We first prove the following minor generalization of results of Matsushita \cite[Proposition~4.1]{Mat15}, \cite[Theorem~1.2]{Mat01}.

\begin{thm}[Fujiki relations, index of the Beauville-Bogomolov form]
	\label{mainfujiki}
	Let $(X,\omega)$ be a $2n$-dimensional, irreducible symplectic variety. The Beauville-Bogomolov form $q_X$ has the following properties:
	\begin{enumerate}
		\item\label{fr} There is a number $c_X\in\R^+$, called \emph{Fujiki constant}, such that for all ${\alpha\in\H^2(X,\,\C)}$ the following, so-called \emph{Fujiki relation}, holds,\footnote{We will recall integration on singular cohomology in Section~\ref{integration}.}
		\[
			c_X\cdot q_{X}(\alpha)^n=\int_X\alpha^{2n}.
		\]
		In particular for classes of Cartier divisors $d=c_1\big(\O_X(D)\big)$, this relates the Beauville-Bogomolov form to the intersection product via $c_X\cdot q_{X}(d)^n=D^{2n}$.
		\item\label{ind} The restriction of $q_{X}$ to $\H^2(X,\,\R)\ra\R$ is a real quadratic form of \emph{index} $\big(3,0,b_2(X)-3\big)$.
	\end{enumerate}
\end{thm}

Part~\eqref{ind} implies the existence of a $q_X$-orthogonal decomposition ${\H^2(X,\,\R)=V_+\oplus V_-}$ with $\dim_{\R}V_+=3$, such that $q_{X}$ is positive definite on $V_+$ and negative definite on $V_-$, see Proposition~\ref{decomposition}.

In the smooth case, these properties were proved by Fujiki and Beauville \cite[Theorem~4.7]{Fuj87}, \cite[th.~5(a)]{Bea83}. Matsushita generalized them to Namikawa symplectic varieties \cite[Theorem~1.2]{Mat01}, \cite[Proposition~4.1]{Mat15}.

Matsushita uses his versions of Theorem~\ref{mainfujiki} to show that fibrations of irreducible symplectic manifolds and Namikawa symplectic varieties have very special properties, \cite[Theorem~2]{Mat99}, \cite[Corollary~1]{Mat00}, \cite[Corollary~1.4]{Mat01}, \cite[Theorem~1.2,~1.10]{Mat15}. As our main results, the following Theorems~\ref{fibration} and \ref{fibration2}, we prove that all these results hold for primitive symplectic varieties and we show, which of them are also true for all irreducible symplectic varieties.

\begin{thm}[Fibrations of irreducible symplectic varieties]
\label{fibration}
  Let $(X,\omega)$ be an irreducible symplectic variety of complex
  dimension $2n$, together with a surjective morphism $f\colon X\ra B$ with connected fibers onto a
  normal, projective variety $B$ with $0<\dim B<2n$. Then the following properties hold.
\begin{enumerate}
\item\label{base} The base variety $B$ is an $n$-dimensional, $\Q$-factorial klt variety with Picard number $\rho(B)=1$.

\item\label{singlocus} The singular locus of $X$ is mapped to a proper closed subset of $B$. In particular, the general fiber is smooth and entirely contained in $X_{\reg}$.

\item\label{genfiber} The general fiber is an $n$-dimensional Abelian variety.
                                                                         
\item\label{equidim} Every fiber component of $f$ is an $n$-dimensional \emph{Lagrangian subvariety}\footnote{We discuss the definition of \emph{Lagrangian subvarieties} in Section~\ref{lagrangiansection}.} of $X$ and does not lie completely in $X_{\sing}$.
\end{enumerate}
\end{thm}

\begin{thm}[Fibrations of primitive symplectic varieties]
\label{fibration2}
  Let $f\colon X\ra B$ be as in Theorem~\ref{fibration}. If $X$ is a primitive symplectic variety, then the following holds.
\begin{enumerate}
\item\label{fano} The base variety $B$ is \emph{Fano}, so $-K_B$ is ample.
\item\label{bsmooth} If $B$ is smooth, then $B\cong\P^n$.
\end{enumerate}
\end{thm}

Both parts of Theorem~\ref{fibration2} were known to hold for all irreducible symplectic manifolds \cite[Theorem~2~(3)]{Mat01}, \cite[Theorem~1.2]{Hwa08}, \cite[Theorem~1.1]{GL14}. It is remarkable that if $X$ is an irreducible symplectic variety, then $K_B$ can vanish, hence part~\ref{fano} can fail for irreducible symplectic varieties, as Example~\ref{ex} shows.

To prove the results on fibrations we work as much as possible on the singular variety $X$. We hope that this will also allow to tackle further questions on Lagrangian fibrations. Open problems are for example a classification of the singular fibers and of the possible base varieties $B$. We still know little about its possible singularities. When $X$ is smooth, it is conjectured that $B$ is also smooth, but so far this got only proved for $\dim X\leq 4$ by Huybrechts--Xu and Ou \cite{HX19}, \cite[Theorem~1.2]{Ou16}. In general $B$ can certainly be singular, even if $X$ is primitive symplectic, \cite[Theorem~1.9]{Mat15}.

\subsection{Outline of the paper}

We present in \ref{extensionsection}--\ref{hodgetheory} our most important methods to prove our main theorems. They are the general extension theorem for differential forms from \cite[Theorem~1.4]{GKKP11}, the existence of terminal models from \cite[Corollary~1.4.3]{BCHM10} and results on the Hodge theory of klt varieties that we explain in detail in \cite{Sch17b}. In particular we will need the Hodge decomposition of $\H^2(X,\,\C)$ and a singular version of the Hodge-Riemann bilinear relations.

After recapitulating well-known properties of symplectic varieties in \ref{propertiessection}, we prove that terminal models of irreducible or primitive symplectic varieties are Namikawa or primitive symplectic, respectively. In \ref{lagrangiansection} we prove in Theorem~\ref{lag} that singular fibers of Lagrangian fibrations are automatically also Lagrangian, implying that Lagrangian fibrations are always equidimensional. This requires a careful revision of the definitions of Lagrangian subvarieties and a subtle inductive argument. Theorem~\ref{lag} will be used to show that the general fiber of a Lagrangian fibration is an Abelian variety, which in fact holds for fibrations of every symplectic variety. Here we make use of a recently proven special case of the Lipman-Zariski conjecture, which makes the proof even easier and more geometric than Matsushita's proof of \cite[Corollary~1.4 (2)]{Mat01}. We explain in Section~\ref{generalizedform} how the Beauville-Bogomolov form $q_X$ is defined in the singular setting in a way that is most suitable for direct computations in $\H^2(X,\,\C)$.

In Section~\ref{proofs} we finally prove our main results. We deduce the Fujiki relation directly from a terminal model, and use our Hodge theoretic methods to calculate the index of $q_X$ by hand similarly like Matsushita. Theorem~\ref{mainfujiki} is essential to prove the rest of Theorem~\ref{fibration} in the spirit of Matsushita's proof of \cite[Corollary~1.4]{Mat01}. Proving that $B$ is $\Q$-factorial becomes particularly involved because handling Weil divisors is much harder in the singular case.  The proof of Theorem~\ref{fibration2} uses the branched covering trick and going over to a terminal model, which we have seen to be primitive symplectic again.

\subsection{Conventions}
We use the terminology of \cite{Har77}, in particular varieties are defined to be irreducible. We denote the smooth and the singular locus of a variety $X$ by $X_{\reg}$ and $X_{\sing}$, respectively. For the terminology of the Minimal Model Program we refer to \cite{KM98}, but we use the sloppy term \emph{klt variety} for a variety $X$ that carries an \emph{effective} Weil divisor $D$, such that the pair $(X,D)$ is klt according to \cite[Definition~2.34]{KM98}.

\subsection{Acknowledgment}

The results of this paper are an improved version of results of my PhD thesis, answering also Questions 6, 8 and Conjecture 7 that were stated in the introduction of the thesis, \cite{Sch17b}. I want to thank my supervisor S.\,Kebekus, as well as K.\,Oguiso, D.\,Greb, P.\,Graf, T.\,Kirschner and B.\,Taji for their advice and fruitful discussions. Special thanks to the anonymous referee for several corrections and improvements.

\section{Methods}

We recall here the extension theorem, terminal models, integration over singular cohomology classes and results on the Hodge theory for klt varieties. We recapitulate well-known properties of symplectic varieties that will
imply that terminal models of irreducible or primitive symplectic varieties are Namikawa or primitive
symplectic, respectively. Then we discuss the notions of Lagrangian subvarieties and fibrations and examine their singular fibers.

 \subsection{Extension Theorem}
\label{extensionsection}
The extension theorem for differential forms \cite[Theorem~1.4]{GKKP11} shows that on a klt variety $X$ the pullback of every reflexive form ${\alpha\in\H^0(X,\,\Omega^{[p]}_X)}$ to a resolution of singularities can be \emph{extended} over the exceptional locus. This allows us to construct pullbacks of reflexive forms along more general morphisms. 

\begin{thm}[Pullbacks of reflexive forms, \protect{\cite[Theorem~1.3]{MR3084424}}]\label{extension}
  Let $f\colon Y\ra X$ be a morphism between normal, complex quasi-projective varieties where $X$ is a klt variety. Then for every $p$ there is a natural pullback morphism $f^*\Omega^{[p]}_X\ra\Omega^{[p]}_Y$, consistent with the natural pullback of Kähler differentials on $X_{\reg}$.\qed
\end{thm}

The consistency with the natural pullback of Kähler differentials is made precise in \cite[Theorem~5.2]{MR3084424}. We will only consider the case when $f\colon Y\ra X$ is surjective. Then this consistency means that the pullback morphism $f^*\Omega^{[p]}_X\ra\Omega^{[p]}_Y$ agrees with the usual pullback of Kähler differentials wherever $X$ and $Y$ are smooth. In other words, for every $\alpha\in\H^0(X,\,\Omega^{[p]}_X)$ there is an $\tilde{\alpha}\in\H^0(Y,\,\Omega^{[p]}_Y)$ that restricts on $Y_{\reg}\cap f^{-1}(X_{\reg})$ to the natural pullback $f^*(\restr{\alpha}{X_{\reg}})$ of Kähler differentials. We call $\tilde{\alpha}$ an \emph{extension} of $f^*{\alpha}$ to $Y$. Pulling back reflexive forms is contravariant functorial with respect to $f$. 

\subsection{Terminal models}
\label{terminalmodelsection}
While not every singular variety admits a crepant resolution, complex projective klt varieties have always a \emph{terminal model}. For canonical singularities a terminal model is a \emph{crepant partial resolution}. This is an easy special case of \protect{\cite[Corollary~1.4.3]{BCHM10}}.

\begin{thm}[Terminal models, \protect{\cite[page~413]{BCHM10}}]
\label{term}
Let $X$ be a complex projective variety with canonical singularities. Then every resolution $\nu\colon\tilde{X}\ra X$ of $X$ factors as ${\tilde{X}\stackrel{\tilde{\pi}}{\lra}Y\stackrel{\pi}{\lra}X}$, where $\tilde{\pi},\pi$ are birational and $Y$ is a $\Q$-factorial variety with at most terminal singularities, such that $\pi^*K_X=K_Y$. We call $Y$ a \emph{terminal model} of $X$. 
\qed
\end{thm}

\subsection{Integration of singular top classes}
\label{integration}

Recall that we integrate a top singular cohomology class $\phi\in\H^{2n}(X,\,\C)$ over an $n$-dimensional, compact, complex variety $X$ as $\int_X\phi\defeq[X]\cap\phi$. Here $[X]\in\H_{2n}(X,\,\Z)$ denotes the canonical fundamental class of $X$, induced by the complex structure on its smooth locus $X_{\reg}$, \cite[Section~2.5]{Sch17b}. This integration commutes with pullbacks by bimeromorphic maps. 

\begin{lem}[\protect{Pullbacks of integrals, \cite[Lemma~22]{Sch17b}}]
	\label{pullbackintegrals}
	If $f\colon X\ra Y$ is a bimeromorphic morphism of compact, complex varieties, then $\int_Y\alpha=\int_Xf^*\alpha$ for all
	$\alpha\in\H^{2n}(Y,\,\C)$.\qed
\end{lem}

\subsection{Hodge Theory for klt varieties}
\label{hodgetheory}

When $X$ is a compact complex variety with rational singularities, Namikawa noted in \cite[p.~143]{Nam06} that for every resolution of singularities $\nu\colon\tilde X\to X$ the pullback morphism $\nu^*\colon\H^2(X,\,\C)\to\H^2(\tilde{X},\,\C)$ is an embedding and therefore the Hodge structure on $\H^2(X,\,\C)$ is pure of weight two. By similar arguments like in \cite[Lemma~2.7]{Kal06} it can be seen that for every reflexive $2$-form $\alpha\in\H^0(X,\,\Omega_X^{[2]})$ for which $f^*\alpha$ extends to a form $\tilde{\alpha}\in\H^0(\tilde{X},\,\Omega^2_X)$ there is a class $\underline{\alpha}\in\H^2(X,\,\C)$ with $f^*\underline{\alpha}=[\tilde{\alpha}]\in\H^2(\tilde{X},\,\C)$. When we have an extension theorem for reflexive forms on $X$ like Theorem~\ref{extension}, it follows that $\alpha\mapsto\underline{\alpha}$ defines an isomorphism $\H^0(X,\,\Omega_X^{[2]})\cong\H^{2,0}(X)$. The author worked this out and drew the following consequences in his thesis.

\begin{thm}[Hodge decomposition of $\H^2(X,\,\C)$, \protect{\cite[Theorem~37]{Sch17b}, \emph{cf.} also \cite[B.2.7--B.2.9]{Kir15}, \cite[Proposition~6.9]{GKP16}}]
	\label{hdec}
	When $X$ is a complex projective klt variety, then we have $\H^{2,0}(X)\cong\H^0(X,\,\Omega_X^{[2]})$ and $\H^{0,2}(X)\cong\H^2(X,\,\O_X)$. For every resolution $\nu\colon\tilde X\to X$ the pullback morphism $\nu^*\colon\H^{a,b}(X)\to\H^{a,b}(\tilde{X})$ is bijective for $(a,b)=(2,0)$ or $(0,2)$ and injective for $(a,b)=(1,1)$.
\end{thm}

\begin{rem}
	\label{notationinclusion}
	Let $X$ be a complex projective klt variety. Every reflexive two-form $\alpha$ on $X$ defines by Theorem~\ref{hdec} a unique cohomology class in $\H^2(X,\,\C)$ that we denote by $\underline{\alpha}$. When $f\colon Y\to X$ is a morphism of complex projective klt varieties and $\tilde{\alpha}$ the extension of $f^*\alpha$ to $Y$ by Theorem~\ref{extension}, then we have $\underline{\tilde{\alpha}}=f^*\underline{\alpha}$ in $\H^2(Y,\,\C)$ because pulling back is a morphism of Hodge structures.
\end{rem}

\begin{cor}[Bilinear relations on klt varieties, \protect{\cite[Corollary~42]{Sch17b}}]
	\label{bilrel}
	Let $X$ be an $n$-dimensional, complex projective klt variety. Every ample class ${a\in\H^2(X,\,\R)}$ induces a sesquilinear form $\psi_{X,a}$ on $\H^k(X,\,\C)$ defined by
	\[(v,w)\mapsto(-1)^{\frac{k(k-1)}{2}}\cdot\int_X v\cup\overline{w}\cup a^{n-k}.\] The so-called \emph{Hodge-Riemann bilinear relation} $i^{p-q}\cdot\psi_{X,a}(v,v)>0$ holds for any non-zero $v$ with ${v\cup a^{n-k+1}=0}$ that can be written as the cup product of classes $v_1,\ldots,v_r\in\H^2(X,\,\C)$ .
	
	We write $p\defeq\sum p_j$, $q\defeq\sum q_j$ and $k\defeq p+q$ where $(p_j,q_j)$ are such that the $\alpha_j$ lie in the $\H^{p_j,q_j}(X)$-part of the Hodge decomposition of $\H^2(X,\,\C)$.
	\qed
\end{cor}

\subsection{Properties of symplectic varieties}
\label{propertiessection}
A symplectic variety has trivial canonical class, as the top exterior power of the symplectic form trivializes the canonical sheaf. It follows that symplectic varieties have canonical singularities \cite[Proposition 1.3]{Bea00}, \cite[Corollary~5.24]{KM98}. Furthermore, Namikawa showed that a symplectic variety has at most terminal singularities if and only if its singular locus is at least of codimension four, \cite[Corollary~1]{Nam01a}.

We point out that Matsushita uses in \cite{Mat01} the property $\h^2(X,\,\O_X)=1$ of irreducible symplectic varieties instead of $\h^0(X,\,\Omega^{[2]}_X)=1$. It follows from Theorem~\ref{hdec} that these two conditions are equivalent by the Hodge symmetry.

\begin{pro}\label{termismatsushita}
	Every terminal model $\pi\colon Y\ra X$ of an irreducible symplectic variety $(X,\omega)$ is a Namikawa symplectic variety with the extension $\omega'\in\H^0(Y,\,\Omega^{[2]}_Y)$ of $\pi^*\omega$ to $Y$ from Theorem~\ref{extension}.
\end{pro}
\begin{proof}
Namikawa proved that $Y$ is again symplectic \cite[Remark~1]{Nam06}. Let $\tilde{\pi}\colon\tilde{X}\ra Y$ be a resolution of $Y$. Then $\h^1(Y,\,\O_Y)=\h^1(\tilde{X},\,\O_{\tilde{X}})=\h^1(X,\,\O_X)=0$ as the irregularity is a birational invariant for varieties with rational singularities and $\h^0(Y,\,\Omega^{[2]}_Y)=\h^0(X,\,\Omega^{[2]}_X)=1$ by Theorem~\ref{extension}. Hence $(Y,\tilde{\omega})$ is Namikawa symplectic by \cite[Corollary~1]{Nam01a}.
\end{proof}

\begin{pro}
\label{termisprimitive}
Every terminal model $\pi\colon Y\ra X$ of a primitive symplectic variety $(X,\omega)$ is primitive symplectic with the extension $\omega'\in\H^0(Y,\,\Omega^{[2]}_Y)$ of $\pi^*\omega$ to $Y$.
\end{pro}
\begin{proof}
Let $g\colon Y'\ra Y$ be quasi-\'{e}tale with $Y'$ normal. Then the Stein factorization of $\pi\circ g$, \cite[Corollary~III.11.5]{Har77}, gives a morphism $\pi'\colon Y'\to X'$ with connected fibers and a finite morphism $h\colon X'\to X$ with $\pi\circ g=h\circ\pi'$. Hence $\pi'$ is birational because $\pi\circ g$ is generically finite. If $g$ is \'{e}tale outside of $V\subset Y'$, then $h$ is \'{e}tale outside of $\pi'(V)$ because $\pi$ and $\pi'$ induce isomorphisms on the function fields. Thus $h$ is a quasi-\'{e}tale. We note that normal, quasi-\'{e}tale covers of klt varieties are also klt varieties by \cite[Proposition~5.20]{KM98}. A resolution $\tilde{X}$ of $X'$ is also a resolution of $Y'$ and Theorem~\ref{extension} implies ${\H^0(Y',\,\Omega_{Y'}^{[p]})\cong\H^0(\tilde{X},\,\Omega_{\tilde{X}}^p)\cong\H^0(X',\,\Omega_{X'}^{[p]})}$ for all $p$. As $X$ is primitive symplectic and pulling back reflexive forms is functorial, the reflexive forms on $Y'$ get generated by $(\pi')^*h^*\omega=g^*\omega'$.
\end{proof}

While we cannot prove that smooth irreducible symplectic varieties are irreducible symplectic manifolds as they were defined by Beauville, we can prove the following Lemma, which is related to \cite[Proposition~A.1]{HN11}.

\begin{lem}
\label{simplyconnected}
Let $(X,\omega)$ be a smooth, irreducible symplectic variety in the sense of Definition~\ref{symplectic}. Then $X$ is either simply connected or an \'{e}tale quotient of an Abelian variety by a finite group of biholomorphic automorphisms.
\end{lem}
\begin{proof}
By the Beauville-Bogomolov decomposition theorem, there is a finite \'{e}tale covering $\pi\colon\hat X\to X$ that splits as a product $\hat{X}\cong T\times X'\times Y$, where $T$ is a complex torus, $X':=\prod_{i=1}^{k} X_i$ is the product of irreducible symplectic manifolds $X_i$, and $Y$ a product of at least three-dimensional Calabi-Yau manifolds.
By going over to a finite \'{e}tale covering of $\hat X$ we may assume $\pi$ to be a Galois covering, so $X\cong \hat X/G$ for $G\subset\Aut\hat X$ a group of biholomorphic automorphisms of $\hat X$ with $|G|=\deg\pi$.

We can identify $\H^0(X,\,\Omega^2_X)$ with the space of $G$-invariant holomorphic $2$-forms $\H^2(\hat{X},\,\Omega^2_{\hat{X}})^G$, hence the latter one is generated by the pullback $\pi^*\omega$ of the symplectic form.
As
$$\h^0(X',\,\Omega^1_{X'})=\h^0(Y,\,\Omega^1_Y)=\h^0(Y,\,\Omega^2_Y)=0,$$ we get by the Künneth formula a decomposition $\pi^*\omega=\pi_T^*\eta+\pi_{X'}^*\omega'$, where $\pi_T\colon\hat X\to T$ and $\pi_{X'}\colon\hat{X}\to X'$ are the projections, and $\eta\in\H^0(T,\,\Omega^2_T)$, $\omega'\in\H^0(X',\,\Omega^2_{X'})$ are $2$-forms on the factors.
As $\pi$ is \'{e}tale, $\pi^*\omega$ is also non-degenerate, hence $Y$ is trivial and the forms $\eta$ and $\omega'$ are non-degenerate as well.

For every $f\in\Aut(\hat X)$ there are automorphisms $g\in\Aut(T)$ and $h\in\Aut(X')$ such that $f=(g,h)$
\cite[p.~8, Lemma]{Bea83a}, so $\pi_T^*\eta$ and $\pi_{X'}^*\omega'$ are $G$-invariant holomorphic $2$-forms on $\hat{X}$.
As $\h^0(X,\,\Omega^2_X)=1$, it follows that either $T$ or $X'$ has to be trivial. If $X'$ is trivial then $X\cong T/G$ and the proof is complete.\footnote{It is unclear if this case can occur. The argument in \cite[Proposition~A.1]{HN11} to exclude the torus factor is incomplete.}


Hence we can assume from now on that $T$ is trivial, so $X\cong X'/G$. In this case, we show that every automorphism $f\in\Aut(X')$ has a fixed point. For this, let $\omega_i$ be the pullbacks of the symplectic forms of the factors $X_i$ to $X'$.
By \cite[p.~762f, prop.~3-4]{Bea83} and the Künneth formula, $\bigoplus_{p\in\N}\H^0(X',\,\Omega^{p}_{X'})$ is generated by the wedge products of the $\omega_i$. In particular $\H^{j,0}(X')=0$ for all odd $j$.
As $(X,\omega)$ is symplectic, $f^*$ preserves a symplectic form on $X'$. By rescaling the $\omega_i$ we may assume it to be the sum of the $\omega_i$.

By \cite[p.~10, b)+c)]{Bea83a} the automorphism $f$ acts on $X'$ by first possibly permuting isomorphic factors $X_i$ and then applying automorphisms $f_i\in\Aut(X_i)$ on each factor.
In particular there is a permutation $\sigma\in\mathcal{S}_k$ such that for all $i$ we have that $f^*\omega_i$ is a multiple of $\omega_{\sigma(i)}$. As $f^*$ preserves the sum of the $\omega_i$, it simply permutes the $\omega_i$, as well as their wedge products.
Therefore $\tr(f^*|_{H^{j,0}(X')})$ is zero for all odd $j$, non-negative for all even $j$ and it equals one for $j=0,2\dim(X)$.
Thus $f$ has a fixed point by the holomorphic Lefschetz fixed point formula, \cite[p.~426]{GH}.
As we assumed $X$ to be smooth, it follows that $G$ acts trivially on $X'$, so $X\cong X'$ is simply connected.
\end{proof}

\subsection{Lagrangian fibrations}
\label{lagrangiansection}

We show in Theorem~\ref{equidimensional} that Lagrangian fibrations are equidimensional because also their singular fibers consist of Lagrangian subvarieties. A straightforward generalization of Matsushita's proof of \cite[Corollary~1]{Mat00}, see Proposition~\ref{classvanishes}, works for all fiber components that do not completely lie in the singular locus. To complete the proof of Theorem~\ref{equidimensional} we use an additional inductive argument to make sure that there are no fiber components in $X_{\sing}$.

\begin{defi}[Lagrangian subvariety and fibration]\label{deflag}

Let $(X,\omega)$ be a symplectic variety and $F\subset X$ a subvariety with $F\not\subset X_{\sing}$ and ${\dim F=\frac{1}{2}\dim X}$. If on every \emph{embedded\footnote{A resolution of singularities $\nu\colon\tilde{X}\to X$ that is an isomorphism over $X_{\reg}$ and for which the \emph{strict transform} $\overline{\nu^{-1}(F\cap X_{\reg})}$ of $F$ is smooth} resolution of singularities} $\nu\colon\tilde{X}\to X$ of $(X,F)$ the extension $\tilde{\omega}$ of $\nu^*\omega$ vanishes on the strict transform of $F$, then we call $F$ a \emph{Lagrangian subvariety} of $X$. A \emph{Lagrangian fibration} of $X$ is a surjective morphism $f\colon X\ra B$ with connected fibers onto a normal, complex projective variety $B$, such that the general fiber is a Lagrangian subvariety of $X$.
\end{defi}

The notion of Lagrangian subvarieties does not depend on the chosen embedded resolution because for two resolutions we can go over to a common resolution. Our definition is consistent with Matsushita's definition, \cite[Definition~1]{Mat00}, \cite[Definition~1.2]{Mat05}. The following criteria can be used to test if a subvariety is Lagrangian.

\begin{lem}[Criteria of Lagrangian subvarieties]
\label{lag}
Let $F$ be a subvariety of a symplectic variety $(X,\omega)$ with $F\not\subset X_{\sing}$ and $\dim F\geq\frac{1}{2}\dim X$. Then the following assertions hold.
\begin{enumerate}
\item\label{part1} $F$ is Lagrangian if and only if $\restr{\omega}{X_{\reg}\cap F_{\reg}}$ vanishes as a Kähler differential on $X_{\reg}\cap F_{\reg}$.
\item\label{part2} $F$ is Lagrangian if ${i^*\underline{\omega}=0\in\H^2(F,\,\C)}$ for the inclusion $i\colon F\inj X$. The converse holds if $F$ is a klt variety.
\end{enumerate}
\end{lem}

\subsubsection*{Proof of \eqref{part1}.}
Let $\nu\colon\tilde{X}\ra X$ be an embedded resolution of $(X,F)$ with strict transform $\tilde{i}\colon\tilde{F}\inj\tilde{X}$ of $F$ and extension $\tilde{\omega}$ of $\nu^*\omega$ to $\tilde{X}$. As $\nu$ is an isomorphism over a dense open subset, the assumption $F\not\subset X_{\sing}$ implies that the conditions $\restr{\omega}{X_{\reg}\cap F_{\reg}}=0$ and $\restr{\tilde{\omega}}{\tilde{F}}=0$ are equivalent. Moreover, linear algebra shows that both latter conditions imply ${\dim F\leq\frac{1}{2}\dim{X}}$ because $\omega$ is non-degenerate on $X_{\reg}$, \cite[Lemma~1.1]{NamikawaDeformationTheory}. Together with the assumption $\dim F\geq\frac{1}{2}\dim{X}$ this implies that the conditions $\restr{\omega}{X_{\reg}\cap F_{\reg}}=0$ and $\restr{\tilde{\omega}}{\tilde{F}}=0$ are both equivalent to $F$ being a Lagrangian subvariety of $X$.
\subsubsection*{Proof of \eqref{part2}.}
By Remark~\ref{notationinclusion} we have $\underline{\tilde{\omega}}=\nu^*\underline{\omega}$ and therefore
\[
\underline{\restr{\tilde{\omega}}{\tilde{F}}}=\tilde{i}^*\nu^*\underline{\omega}=\nu|_{\tilde{F}}^*i^*\underline{\omega}\in\H^2(\tilde{F},\,\C).
\]
Hence the condition $i^*\underline{\omega}=0$ implies $\restr{\tilde{\omega}}{\tilde{F}}=0$, which, like in part~\eqref{part1}, under the given assumptions is equivalent to $F$ being Lagrangian. The converse implication holds if $F$ is a klt variety because then ${\nu|_{\tilde{F}}^*\colon\H^2(F,\,\C)\ra\H^2(\tilde{F},\,\C)}$ is injective by Theorem~\ref{hdec}.\qed

The following result is a straightforward generalization of \cite[Corollary~1]{Mat00}.

\begin{pro}
\label{classvanishes}
Let $f\colon X\ra B$ be a Lagrangian fibration of a symplectic variety $(X,\omega)$. Then for every fiber $F$ of $f$ the pullback $i^*\underline{\omega}$ of the symplectic class vanishes in $\H^2(F,\,\C)$, where $i\colon F\ra X$ denotes the inclusion. Thus also $i'^*\underline{\omega}=0\in\H^2(F',\,\C)$ for every fiber component $i'\colon F'\ra X$ of $F$.
\end{pro}
\begin{proof}
The Leray spectral sequence for $f$ and the sheaf $\O_X$ gives the map $d_2\colon \H^2(X,\,\O_X)\ra\H^0(B,\,R^2f_*\O_X)$. The class $\underline{\overline{\omega}}$ lies in $\H^2(X,\,\O_X)$. Let $i\colon F\ra X$ be the fiber over a point $p\in B$ with residue field $\C(p)$. Then the proper base change map $R^2f_*\O_X\otimes_{\O_X}\C(p)\stackrel{\sim}{\lra}\H^2(F,\,\O_F)$ sends $d_2(\underline{\overline{\omega}})\otimes1$ to $i^*\overline{\underline{\omega}}\in\H^2(F,\,\O_F)$.

As the general fiber is Lagrangian, not completely contained in $X_{\sing}$ and has canonical singularities by Lemma~\ref{singfiber}, the section $d_2(\underline{\overline{\omega}})$ vanishes at the general point $p\in B$ by Lemma~\ref{lag}. Therefore this section is torsion in $R^2f_*\O_X$. However, as $X$ is symplectic, the canonical sheaf of $X$ is trivial, so $\omega_X\cong\O_X$. Hence the sheaf $R^2f_*\O_X$ is torsion-free by Koll\'{a}r's torsion-freeness Theorem, which also holds in the singular case, see Theorem~\ref{torsionfree} in the appendix. Thus ${d_2(\underline{\overline{\omega}})=0}$, so for every fiber $i\colon F\ra X$ the class $i^*\underline{\omega}$ vanishes. Moreover $i'^*\underline{\omega}=0$ because $i'$ factors through $i$.
\end{proof}

\begin{thm}
\label{equidimensional}
Let $f\colon X\ra B$ be a Lagrangian fibration of a symplectic variety $(X,\omega)$. Then every component of each fiber of $f$ does not lie completely in $X_{\sing}$ and is Lagrangian. In particular $f$ is equidimensional. 
\end{thm}
\begin{proof}
We prove this theorem by induction over the dimension of $X$. If $X$ is a surface, all fibers of $f$ are curves. As $X$ is normal, it has only isolated singularities and no fiber component can lie in $X_{\sing}$. By Proposition~\ref{classvanishes} the symplectic class vanishes on every fiber component $F'$. Hence $F'$ is a Lagrangian subvariety of $X$ by Lemma~\ref{lag} and the theorem is proven for $\dim X=2$.

We assume now $2n\defeq\dim X>2$ and that the theorem is already proven for all lower dimensional symplectic varieties. Let $i'\colon F'\inj X$ be a component of a fiber $F$ of $f$. Let $Z$ be an irreducible component of $X_{\sing}$.  We consider the restriction $\restr{f}{Z}$ and the lifted morphism $\overline{f}\colon\overline{Z}\ra\overline{f(Z)}$ between the normalizations. Then by \cite[Theorem~3.1]{Mat15} the normalization $\overline{Z}$ is a symplectic variety and $\overline{f}$ is a Lagrangian fibration. By the induction hypothesis, we know that $\overline{f}$ is equidimensional, so every fiber of $\overline{f}$ has dimension at most $n-1$. Hence the dimension of all fibers of $\restr{f}{Z}$ is also at most $n-1$. However, by the fiber dimension theorem, \cite[Theorem~1.25]{Shaf94}, we have $\dim F'\geq n$, so $F'\not\subset Z$ and $F'$ cannot lie completely in $X_{\sing}$. Now Proposition~\ref{classvanishes} implies $i'^*\underline{\omega}=0$, so $F'$ is by Lemma~\ref{lag} an $n$-dimensional Lagrangian subvariety of $X$. This completes the induction.
\end{proof}

\section{Generalized Beauville-Bogomolov form}
\label{generalizedform}

Namikawa defined a generalized Beauville-Bogomolov form on every Namikawa symplectic variety $X$ by pulling back everything to a resolution of singularities $\tilde{X}$ and then calculating Beauville's formula \cite[p.~772]{Bea83}. We prefer working with Kirschner's idea to calculate all integrals directly on the singular cohomology of $X$, which makes explicit calculations easier. Both approaches are equivalent, compatible with pullbacks and, after a suitable normalizing, define for every irreducible symplectic variety a uniquely determined quadratic form $q_X$ on $\H^2(X,\,\C)$. 

\subsection{Definitions of the Beauville-Bogomolov form}
\begin{defi}
	[Beauville-Bogomolov form, cf.\,\protect{\cite[Theorem~8~(2)]{Nam01}}]
	\label{nbbform}
	Let $(X,\omega)$ be a $2n$-dimensional, irreducible symplectic variety with a resolution $\nu\colon\tilde{X}\ra X$ and extension $\tilde{\omega}\in\H^0(\tilde{X},\,\Omega_{\tilde{X}}^2)$ of $\nu^*\omega$ to $\tilde{X}$. We consider the pullbacks $\nu^*\alpha\eqdef\tilde{\alpha}$ of classes $\alpha\in\H^2(X,\,\C)$ as the classes of two-forms. Then the \emph{Beauville-Bogomolov form} on $X$ is the quadratic form $q_{X,\omega}\colon\H^2(X,\,\C)\ra\C$ with
\[
	q_{X,\omega}(\alpha)\defeq \frac{n}{2}\int_{\tilde{X}}(\tilde{\omega}\overline{\tilde{\omega}})^{n-1}\tilde{\alpha}^2+(1-n)\Big(\int_{\tilde{X}}\tilde{\omega}^n\overline{\tilde{\omega}}^{n-1}\tilde{\alpha}\Big)\cdot\Big(\int_{\tilde{X}}\tilde{\omega}^{n-1}\overline{\tilde{\omega}}^n\tilde{\alpha}\Big)
\]
	for all $\alpha\in\H^2(X,\,\C)$. The products and powers denote wedge products of forms.
\end{defi}
The form $q_{X,\omega}$ does not depend on the chosen resolution. Given two resolutions $\nu,\,\nu'$ of $X$, this can be seen by going over to a common resolution factoring through $\nu$ and $\nu'$.

A resolution $\tilde{X}$ does not need to be symplectic again. Note that thus we cannot trivially deduce all properties of $q_{X,\omega}$ from the smooth case by considering $\tilde{X}$.

\begin{defi}[Beauville-Bogomolov form, \protect{cf.\,\cite[Notation~3.2.1.]{Kir15}}]\label{bbform}
\label{kirschnerform}
  Let $X$ be a $2n$-dimensional, complex projective variety. We define for every
  class ${w\in\H^2(X,\,\C)}$ a quadratic form $q_{X,w}\colon\H^2(X,\,\C)\ra\C$ with
  \[
    q_{X,w}(v)\defeq \frac{n}{2}\int_X(w\overline{w})^{n-1}v^2+(1-n)\Big(\int_Xw^n\overline{w}^{n-1}v\Big)\cdot\Big(\int_Xw^{n-1}\overline{w}^nv\Big)
\]
  for all $v\in\H^2(X,\,\C)$. Here the products and powers denote the cup product in the cohomology ring $\H^*(X,\,\C)$. Using our notation from Remark~\ref{notationinclusion}, for an irreducible symplectic
  variety $(X,\omega)$ the quadratic form $q_{X,\underline{\omega}}$ is called
  \emph{Beauville-Bogomolov form} on $X$. 
\end{defi}

In order to obtain a uniquely determined quadratic form $q_X$ on an irreducible symplectic variety $X$, it is convenient to normalize the symplectic class. The following Lemma ensures that the definition will be well defined.

\begin{lem}[Existence and uniqueness of the normalizing]
	\label{norming}
	Let $X$ be a $2n$-dimensional, irreducible symplectic variety $X$. For all $v\in\H^2(X,\,\C)$ we denote ${I(v)\defeq\int_{X}(v\overline{v})^n}$. Every symplectic class ${\underline{\omega}\in\H^2(X,\,\Omega^{[2]}_X)\setminus\{0\}}$ induces a \emph{normalized class} ${w\defeq I(\underline{\omega})^{-\frac{1}{2n}}\cdot \underline{\omega}}$ with $I(w)=1$. The corresponding Beauville-Bogomolov form $q_{X,w}$ does not depend on the choice of $\underline{\omega}$.
\end{lem}

\begin{proof}
	As $X$ is irreducible symplectic, every symplectic class $\alpha\in\H^0(X,\,\Omega^{[2]}_X)\setminus\{0\}$ on $X$ differs only by a constant factor from the class $\omega$, say $\alpha=c\omega$ for $c\in\C^{\times}$. One computes easily that the normalized classes of $c\omega$ and $\omega$ only differ by a factor $\frac{c}{|c|}$ of absolute value one. We see directly from Definition~\ref{kirschnerform} that this does not affect $q_{X,w}$.\qedhere
\end{proof}

\begin{defi}[Normalized Beauville-Bogomolov form]
	\label{normedform}
	Let $X$ be an irreducible symplectic variety. The \emph{normalized Beauville-Bogomolov form} is defined as $q_X\defeq q_{X,w}$ for any class $w\in\H^{2,0}(X)$ with $I(w)=1$.
\end{defi}

\begin{rem}[The symmetric bilinear form]
\label{bilform}
The Beauville-Bogomolov form $q_X$ on an irreducible symplectic variety $(X,\omega)$ is induced by a symmetric bilinear form. Like Matsushita we also denote it by $q_X$, but with two arguments. We have the usual formula $q_X(a,b)=\frac{1}{2}\big(q_X(a+b)-q_X(a)-q_X(b)\big)$ for all $a,b\in\H^2(X,\,\C)$.
\end{rem}

\subsection{Equivalence of the definitions and behaviour under pullbacks}

To relate Namikawa's and Kirschner's definitions of the Beauville-Bogomolov form, we prove in the following lemma that the Definitions~\ref{bbform} and \ref{normedform} behave well under pullbacks along birational morphisms.

\begin{lem}\label{qpullback}
  Let ${\pi\colon Y\ra X}$ be a birational morphism from a normal, complex projective variety $Y$ to a symplectic variety $(X,\omega)$ with an extension $\omega'\in\H^0(Y,\,\Omega^{[2]}_Y)$ of
  $\pi^*\omega$. Then $\underline{\omega}$ is normalized if and only if $\underline{\omega'}$ is normalized and we have 
  \begin{align*}
  	q_{X,\underline{\omega}}=q_{Y,\underline{\omega'}}\circ\pi^*\quad\text{and}\quad q_X=q_Y\circ\pi^*.
  \end{align*}
\end{lem}
\begin{proof}
We noted in Lemma~\ref{pullbackintegrals} that the integration of top cohomology
  classes is compa\-tible with pullbacks under birational morphisms. For every
  $v,w\in\H^2(X,\,\C)$ we can apply this to the occurring integrals in Definition~\ref{bbform}
  of the Beauville-Bogomolov form, such that we get the relation $q_{X,w}(v)=q_{Y,\pi^*w}(\pi^*v)$ for all $v,w\in\H^2(X,\,\C)$. Together with $\pi^*\underline{\omega}=\underline{\omega'}\in\H^2(Y,\,\C)$ by Remark~\ref{notationinclusion} this shows the first equality. Similarly we see $I(\underline{\omega})=I(\pi^*\underline{\omega})=I(\underline{\omega'})$, which concludes the proof.
\end{proof}

\begin{cor}[Equivalence of the definitions, \protect{cf.~\cite[Proposition~3.2.15]{Kir15}}]
	On every irreducible symplectic variety $(X,\omega)$ Namikawa's and Kirschner's definitions of the Beauville-Bogomolov form are equivalent in terms of $q_{X,\omega}=q_{X,\underline{\omega}}$.
\end{cor}
\begin{proof}
	Let $\nu\colon\tilde{X}\ra X$ be a resolution and $\tilde{\omega}$ the extension of $\nu^*\omega$ to $\tilde{X}$. Using the notation of Kirschner's Definition~\ref{bbform}, we can write Namikawa's Definition~\ref{nbbform} as ${q_{X,\omega}=q_{\tilde{X},\tilde{\underline{\omega}}}\circ\nu^*}$. This equals $q_{X,\underline{\omega}}$ by Lemma~\ref{qpullback}.
\end{proof}

\section{Proofs of the main results}
\label{proofs}
\subsection{Fujiki relations and the index of the Beauville-Bogomolov form}
\label{firstproof}

A natural approach to prove Theorem~\ref{mainfujiki} would be to pass to a resolution of singularities, but this might not be a symplectic variety anymore. Instead, we deduce the Fujiki relation from a terminal model, which we have seen to be a Namikawa symplectic variety. Part~\eqref{ind} of Theorem~\ref{mainfujiki} will follow directly from Proposition~\ref{decomposition} that contains more information about the decomposition of $\H^2(X,\,\C)$. The Hodge theory of $\H^2(X,\,\C)$ and the Fujiki relations allow us to calculate the index by hand.

\begin{proof}[Proof of Theorem~\ref{mainfujiki}, part~\eqref{fr}]
Let $X$ be a $2n$-dimensional, irreducible symplectic variety with normalized Beauville-Bogomolov form $q_X$ on $\H^2(X,\,\C)$ as in Definition~\ref{normedform}. We need to show that there is a constant $c_X\in\R^+$, such that for all $v\in\H^2(X,\,\C)$ we have the Fujiki relation $c_X\cdot q_X(v)^n=\int_Xv^{2n}$.

We write $q_X=q_{X,\omega}$ for a normalized symplectic class $\omega$ with ${\int_X(\underline{\omega\overline{\omega}})^n=1}$. We take a terminal model $\pi\colon Y\ra X$
of $X$ together with an extension
$\omega'\in \H^2(Y,\,\Omega_Y^{[2]})$ of $\pi^*\omega$ from Theorem~\ref{extension}. Then $(Y,\omega')$ is a Namikawa symplectic
variety by Proposition~\ref{termismatsushita}. We have $q_Y=q_{Y,\omega'}$ and $q_X=q_Y\circ\pi^*$ by Lemma~\ref{qpullback}. Now we apply \cite[Theorem~1.2]{Mat01} to $Y$ to get a $c_Y\in\R^+$, such that for all
$u\in\H^2(Y,\,\C)$ we get $c_Y\cdot q_Y(u)^n=\int_Yu^{2n}$. Then by the Lemmas \ref{qpullback} and \ref{pullbackintegrals} we get for every $v\in\H^2(X,\,\C)$:
\[
	c_Y\cdot q_X(v)^n=c_Y\cdot q_Y(\pi^*v)^n=\int_Y(\pi^*v)^{2n}=\int_Y\pi^*(v^{2n})=\int_Xv^{2n}
\]
Hence we have proved the Fujiki relation for the constant $c_X\defeq c_Y$.
\end{proof}

\begin{pro}[Index of the Beauville Bogomolov form $q_X$]
\label{decomposition}
Let $(X,\omega)$ be a $2n$-dimensional, irreducible symplectic variety with normalized Beauville-Bogomolov form $q_X=q_{X,w}$ for $w\in\H^2(X,\,\C)$, and let $a\in\H^2(X,\,\R)$ be an ample class.

Restricting $q_X$ gives a real quadratic form $\H^2(X,\,\R)\ra\R$ and we get the $q_X$-orthogonal decomposition $\H^2(X,\,\R)=V_+\oplus V_-$, where $V_+\defeq\left\langle w+\overline{w},\,iw-i\overline{w},\,a\right\rangle_{\R}$ is a 3-dimensional space on which $q_X$ is positive definite and
\[V_-\defeq a^{\perp}\cap\H^{1,1}(X)\cap\H^2(X,\,\R)=\{d\in\H^{1,1}(X)\cap\H^2(X,\R)\mid q_X(d,a)=0\}\]
is a space on which $q_X$ is negative definite.
\end{pro}

\subsubsection*{Proof that restricting $q_X$ gives a real form.}
An easy computation for arbitrary ${v\in\H^2(X,\C)}$ shows that ${q_{X,w}\big(\overline{v}\big)=\overline{q_{X,w}(v)}}$, using that the integrals are compatible with complex conjugation. Hence restricting gives a real quadratic form $\H^2(X,\,\R)\ra\R$.

\subsubsection*{\protect{Proof of $V_+\perp V_-$ and the positivity of $q_X$ on $V_+$ \emph{(}cf.~\cite[(2.1)]{Mat01}\emph{)}}.}
We can use $\int_X(w\overline{w})^n=1$ and the formula from Remark~\ref{bilform} to calculate the Beauville-Bogomolov bilinear form $q_X$ on the classes $w+\overline{w},\,iw-i\overline{w},\,a$. The Hodge decomposition from Theorem~\ref{hdec} allows us to calculate the occurring integrals via the usual type considerations, which gives the relations $q_X(w+\overline{w})=q_X(iw-i\overline{w})=1$, $q_X(w+\overline{w},\,iw-i\overline{w})=0$, $q_X(w+\overline{w},\,v)=q_X(iw-i\overline{w},\,v)=0$ for all $v\in\H^{1,1}(X)$, and $q_X(a)=\frac{n}{2}\int_X(w\overline{w})^{n-1}a^2$.

Therefore $w+\overline{w},\,iw-i\overline{w}$ and $a$ are orthogonal with respect to $q_X$. The two subspaces $V_+$ and $V_-$ of $\H^2(X,\,\C)$ are likewise orthogonal. The restriction $\restr{q_X}{V_+}$ is diagonal with eigenvalues $1,\,1,\,q_X(a)$. As $w^{n-1}\cup a^3=0$ by type considerations, the Hodge-Riemann bilinear relations, Corollary~\ref{bilrel}, show that $q_X(a)=\frac{n}{2}\psi_{X,a}(w^{n-1},w^{n-1})$ is positive and therefore $q_X$ is positive definite on $V_+$.

\subsubsection*{Proof that $V_+\oplus V_-=\H^2(X,\,\R)$.}
The classes $w+\overline{w},\,iw-i\overline{w}$ and $a$ are obviously real. On the other hand, as $X$ is irreducible symplectic, $\{w,\overline{w}\}$ is a $\C$-basis of $\H^{2,0}(X)\oplus\H^{0,2}(X)$. Hence every real class $v\in\H^2(X,\,\R)$ can be decomposed as ${v=\alpha w+v^{1,1}+\beta\overline{w}}$ with $\overline{\alpha}=\beta\in\C$ and ${v^{1,1}=\overline{v^{1,1}}\in\H^{1,1}(X)\cap\H^2(X,\,\R)}$. We note ${q_X(v^{1,1}-\lambda a,a)=0}$ for ${\lambda\defeq\frac{q_X(v^{1,1},a)}{q_X(a)}\in\R}$. This gives the decomposition
\[
v=\underbrace{Re(\alpha)\cdot(w+\overline w)+Im(\alpha)\cdot(iw-i\overline{w})+\lambda a}_{\in V_+}+\underbrace{v^{1,1}-\lambda a}_{\in V_-}
\]

Hence we get ${V_++ V_-=\H^2(X,\,\R)}$. The directness of this sum decomposition will follow from the negativity of $q_X$ on $V_-$.

\subsubsection*{Proof of the negativity of $q_X$ on $V_-$.}

(\emph{Cf.~\cite[(2.4)]{Mat01}}) Every class $d$ in the space $V_-$ is a real class of type $(1,1)$ with $q_X(d,a)=0$. We need to show $q_X(d)<0$ if $d\neq0$. We use the Fujiki relations on $X$ to calculate for every $t\in\R$:
\[
\int_X(td+a)^{2n}=c_X\cdot q_X(td+a)^n=c_X\cdot \big(t^2q_X(d)+\underbrace{tq_X(d,a)}_{=0}+q_X(a)\big)^n.
\]
Comparing the $t$ and $t^2$-terms on both sides yields $d\cup a^{2n-1}=0$ and
\[
(2n-1)\cdot\int_X(d^2\cup a^{2n-2})=c_X\cdot q_X(d)\cdot q_X(a)^{n-1}.
\]
We have $d=\overline{d}$ because $d$ is real. Therefore the Hodge-Riemann bilinear relations, Corollary~\ref{bilrel}, imply $\int_X(d^2\cup a^{2n-2})=\psi_{X,a}(d,d)<0$ and hence $q(d)<0$.\qed

\subsection{Fibrations of irreducible symplectic varieties}
\label{secondproof}

	Let $(X,\omega)$ be an irreducible symplectic variety of complex
	dimension $2n$, together with a surjective morphism $f\colon X\ra B$ with connected fibers onto a
	normal, complex projective variety $B$ with dimension $0<\dim B<2n$. We subdivide the proof of Theorem~\ref{fibration} into seven steps, which we will prove in the following order:

\begin{itemize}
	\item[\mylabel{dimb}{1a}] $\dim B=n$.
	\item[\mylabel{genlagrangian}{4a}] The general fiber is a Lagrangian subvariety of $X$.
	\item[\mylabel{lagrangian}{4b}] Every fiber component of $f$ is an $n$-dimensional Lagrangian subvariety of $X$ and does not lie completely in $X_{\sing}$.
	\item[\mylabel{smooth}{3a}] The general fiber is smooth.
	\item[\mylabel{abelian}{3b}] The general fiber is an Abelian variety.
	\item[\mylabel{fsing}{2}] $X$ is smooth along the general fiber and $f(X_{\sing})\subset B$ is a proper closed subset.
	\item[\mylabel{qfac}{1b}] $B$ is a $\Q$-factorial klt variety with $\rho(B)=1$.
\end{itemize}

\begin{proof}[Proof of Theorem~\ref{fibration}]
For the proof we choose very ample divisors $A$ on $X$ and $H$ on $B$. We denote their first Chern classes by $a\in\H^2(X,\,\C)$ and $h\in\H^2(B,\,\C)$. For every class ${d=c_1\big(\O_B(D)\big)}$ that comes from a Cartier divisor $D$ on $B$  we get $(f^*D)^{2n}=0$. Moreover, the Fujiki relation $(f^*D)^{2n}=c_X\cdot q_X(f^*d)^n$ gives
\begin{equation}
\label{qfd}
	q_X(f^*d)=0.
\end{equation}

\subsubsection*{Proof of \eqref{dimb}.} (\emph{Cf.~\cite[(3.2)]{Mat01}}) We calculate for every $t\in\R^+$
\[
(A+tf^*H)^{2n}=c_X\cdot q_X(a+tf^*h)^n=c_X\cdot\big(q_X(a)+2tq_X(a,f^*h)\big)^n
\]
by Equation~\eqref{qfd}. We see from the $t$-coefficient that $q_X(a,f^*h)$ is positive because $A$ is ample:
\begin{equation}
\label{qafh}
q_X(a,f^*h)=\frac{A^{2n-1}.f^*H}{c_X\cdot q_X(a)^{n-1}}>0.
\end{equation}
Therefore the $t^k$-coefficient $\binom{2n}{k}\cdot A^{2n-k}.(f^*H)^k$ vanishes for $k>n$, but survives for $k=n$:
\[
\binom{2n}{n}\cdot A^n.(f^*H)^n=2^nc_X\cdot q_X(a,f^*h)^n>0.
\]
This shows $\dim B=n$.
\subsubsection*{Proof of \eqref{genlagrangian}.} (\emph{Cf.~\cite[(3.3)]{Mat01}}) 
We consider a general fiber $F$ over a smooth point $b\in B$ and the embedding ${i\colon F\inj X}$. By Lemma \ref{singfiber} in the appendix the general fiber $F$ fulfills ${F_{\reg}=F\cap X_{\reg}}$ and has at most canonical singularities. Therefore, considering the symplectic form $\omega$ as a holomorphic form on the smooth locus $X_{\reg}$, we can restrict $\omega$ to $F\cap X_{\reg}$ and obtain a reflexive form on $F$, which we denote by ${\restr{\omega}{F}\in\H^2(F,\,\Omega^{[2]}_F)}$. By Lemma~\ref{lag} we need to prove $\restr{\omega}{F}=0$ to show that $F$ is a Lagrangian subvariety of $X$.

We note $\underline{\omega|_F}=i^*w$ for $w\defeq\underline{\omega}$ and that $i^*a$ is an ample class on $F$. The idea is now to show
\begin{equation}
\label{intzero}
\int_Fi^*(w\overline{w}a^{n-2})=0.
\end{equation}
This integral equals $\psi_{F,i^*a}(i^*w,i^*w)$ and by the Hodge-Riemann bilinear relations, Corollary~\ref{bilrel}, it vanishes if and only if $\restr{\omega}{F}=0$, otherwise it is positive.

As $X,F,B$ are compact, complex varieties, they have canonical fundamental classes $[X]\in\H_{4n}(X,\,\Z)$, $[F]\in\H_{2n}(F,\,\Z)$, $[B]\in\H_{2n}(B,\,\Z)$. By definition, we have $H^n=[B]\cap h^n$, so the class ${[B]^{\vee}\defeq\frac{1}{H^n}h^n\in\H^{2n}(B,\,\Z)}$ is dual to $[B]$.

If $X,F,B$ are smooth, Poincar\'{e} duality gives the equality
\begin{equation}
\label{fiberdual}
[X]\cap f^*\big([B]^{\vee}\big)=i_*[F]
\end{equation}
in $\H^{2n}(X,\,\Z)$. Going over to a commutative diagram of resolutions and doing calculations with the projection formula shows that equation~\eqref{fiberdual} holds also in the singular case, \cite[Lemma~63]{Sch17b}. We use this to calculate the integral~\eqref{intzero}.
\begin{align*}
\int_Fi^*(w\overline{w}a^{n-2})&=[F]\cap i^*(w\overline{w}a^{n-2})\\
&=i_*[F]\cap w\overline{w}a^{n-2}&\text{Projection formula}\\
&=\frac{1}{H^n}\big([X]\cap(f^*h)^n\big)\cap w\overline{w}a^{n-2}&\text{By equation~\eqref{fiberdual}}\\
&=\frac{1}{H^n}[X]\cap\big(w\overline{w}a^{n-2}\cup(f^*h)^n\big)&\textbf{\cite[Theorem~VI~5.2.(3)]{Bre97b}}\\
&=\frac{1}{H^n}\int_Xw\overline{w}a^{n-2}(f^*h)^n
\end{align*}
We calculate the latter integral as a term of the following Fujiki relation for ${s,t\in\R^+}$.
\begin{align*}
	\int_X(w+\overline{w}+sa+tf^*h)^{2n}&=c_X\cdot q_X(w+\overline{w}+sa+tf^*h)^n\\
	&=c_X\cdot\big(q_X(w+\overline{w})+s^2q_X(a)+2stq_X(a,f^*h)\big)^n	
\end{align*}
In the last step we used the vanishing of $q_X(f^*h)$ by equation~\eqref{qfd} and the vanishing of $q_X(w+\overline{w},a)$ and $q_X(w+\overline{w},f^*h)$ that follows directly from the definition of the Beauville-Bogomolov form via type considerations. Comparing the $s^{n-2}t^n$-terms gives $\int_Xw\overline{w}a^{n-2}(f^*h)^n=0$ and therefore $\restr{\omega}{F}=0$.

\subsubsection*{Proof of~\eqref{lagrangian}.}
By part~\eqref{genlagrangian} $f$ is a Lagrangian fibration. Hence part~\eqref{lagrangian} follows from Theorem~\ref{equidimensional}.

\subsubsection*{Proof of \eqref{smooth}.}
Let $\fI$ denote the ideal sheaf corresponding to $F$ on $X$ and $\fN_{F/X}\defeq\fH om(\fI/\fI^2,\O_F)$ the normal sheaf of $F$ in $X$. As $F$ is a general fiber, we have $\fN_{F/X}|_{F_{\reg}}\cong\fN_{F_{\reg}/X_{\reg}}$ by Lemma~\ref{singfiber}. As $F$ is Lagrangian, contracting with the symplectic form gives the well-known isomorphism $\fN_{F_{\reg}/X_{\reg}}\cong\Omega^1_{F_{\reg}}$.

The differential $df$ induces a morphism $\fT_{X_{\reg}}|_{F_{\reg}}\to T_bB\times F_{\reg}$, where we consider $T_bB\times F_{\reg}$ as a free sheaf of rank $n$ on $F_{\reg}$.
The kernel of $df$ is exactly $\fT_{F_{\reg}}$ and $df$ is surjective because of ${\dim B+\dim F=\dim X}$.
The exact sequence of the tangent complex, \cite[page~182]{Har77}, gives us $\fN_{F_{\reg}/X_{\reg}}\cong\sfrac{\fT_{X_{\reg}}|_{F_{\reg}}}{\fT_{F_{\reg}}}\cong T_bB\times F_{\reg}$, so the normal bundle $\fN_{F_{\reg}/X_{\reg}}$ is also free of rank $n$.

Being the dual of a coherent sheaf, $\fN_{F/X}$ is reflexive, \cite[Corollary~1.2]{MR597077}. As $F$ is normal, the isomorphism $\fN_{F/X}|_{F_{\reg}}\stackrel{\sim}{\lra}\Omega^1_{F_{\reg}}$ extends uniquely to an isomorphism $\fN_{F/X}\stackrel{\sim}{\lra}\Omega^{[1]}_F$ of free sheaves on $F$, \cite[Proposition~1.6]{MR597077}.

Therefore the tangent sheaf $\fT_F=\big(\Omega^{[1]}_F\big)^{\ast}$ is globally free and $F$ is smooth by the Lipman-Zariski conjecture, which is already proven for the singularities of the minimal model program, \cite[Theorem~6.1]{GKKP11}, \cite[Corollary~1.3]{GK14}, \cite[Theorem~1.1]{Dru14}.

\subsubsection*{Proof of \eqref{abelian}.} As we showed in part~\eqref{smooth} that the general fiber $F$ of $f$ is a compact Kähler manifold with globally free tangent sheaf, we may directly conclude by \cite[Corollary~2]{Wan54} that $F$ is an Abelian variety. A more algebraic argument is to use the freeness of $\Omega^1_F$, so $\h^0(F,\,\Omega^1_F)=\dim F$. The Kodaira dimension $\kappa(F)$ is zero because $K_F=0$. Then by \cite[Corollary~2]{Kaw81} the Albanese morphism $\alpha\colon F\ra\Alb(F)$ is a birational morphism to the Abelian variety $\Alb(F)$. In fact $\alpha$ is an isomorphism because $K_F$ is numerically effective.

\subsubsection*{Proof of \eqref{fsing}.}
As $X_{\sing}$ is closed and $f$ is proper, $f(X_{\sing})$ is closed too. By part~\eqref{smooth} the general fiber $F$ of $f$ is smooth. It is completely contained in $X_{\reg}$ by Lemma~\ref{smoothfiber} in the appendix. Thus $f(X_{\sing})$ is a proper closed subset of $B$.

\subsubsection*{Proof of \eqref{qfac}.} We may assume here that $X$ is Namikawa symplectic, as for the following argument we can go over to a terminal model of $X$. Let $D$ be a Weil divisor on $B$ and $X_0\defeq f^{-1}(B_{\reg})$. As $D$ is Cartier on $B_{\reg}$, we can consider the pullback $f|^*_{X_0}(\restr{D}{B_{\reg}})$. This induces a divisor $\tilde{D}$ on $X$ by taking the closures of its components. Then $m\tilde{D}$ is Cartier for an $m\in\Np$ because $X$ is $\Q$-factorial, and we get a class $d\defeq\frac{1}{m}c_1\big(\O_X(m\tilde{D})\big)$ in $\H^2(X,\C)$.

We have $\codim(X\setminus X_0)=\codim B_{\sing}\geq2$ because $f$ is equidimensional by part~\eqref{lagrangian} and $B$ is normal. Hence no component of $\tilde{D}$ or $D$ lies outside of $X_0$ or $B_{\reg}$, respectively. This implies $f_*\tilde{D}=D$ as Weil divisors and if $D$ is Cartier also ${\tilde{D}=f^*D}$ and $d=f^*c_1\big(\O_B(D)\big)$. 

Like in \cite[Step~4]{Mat99}, we calculate now for $t\in\R$ the Fujiki relation
\begin{align*}
(\tilde{D}&+tf^*H+sA)^{2n}\\
&=c_X\cdot\big(q_X(d)+s^2q_X(a)+2tq_X(d,f^*h)+2sq_X(d,a)+2stq_X(f^*h,a)\big)^n,
\end{align*}
where we used $q_X(f^*h)=0$, and compare coefficients. Considering the constant term, of $t^n$ and of $t^{n-1}s^{n-1}$ gives constants $c_1,\ldots,c_4\in\Q^+$ with
\begin{align*}
\tilde{D}^{2n}&=c_1\cdot q_X(d)^n\\
\tilde{D}^n.(f^*H)^n&=c_2\cdot q_X(d,f^*h)^n\\
\tilde{D}^2.(f^*H)^{n-1}.A^{n-1}&=c_3\cdot q_X(d)q_X(f^*h,a)^{n-1}\\
&+c_4\cdot q_X(d,f^*h)q_X(d,a)q_X(f^*h,a)^{n-2}.
\end{align*}
We show that all three intersection products vanish. For this we may substitute the factors $H$ by general hyperplane sections $H_i$ of $B$. By Bertini's Theorem $H_1,\ldots,H_{n-1}$ intersect in a smooth curve $C\subset B_{\reg}$ that meets $D$ transversally in finitely many points in $B_{\reg}$, \cite{Aki51}. Thus the intersection products $\tilde{D}^n.(f^*H)^n$ and $\tilde{D}^2.(f^*H)^{n-1}.A^{n-1}$ can be calculated over $B_{\reg}$. As $X$ is Cohen Macaulay and $f$ equidimensional, the restriction $\restr{f}{X_0}$ is flat, \cite[Theorem~18.16]{E95}. By \cite[Proposition~2.5d]{Fulton98} we can calculate
\begin{align*}
\tilde{D}^n.(f^*H)^n&=f^*(D|_{B_{\reg}}^n.H|_{B_{\reg}}^n)=0,\\
\tilde{D}^2.(f^*H)^{n-1}.A^{n-1}&=f^*(D|_{B_{\reg}}^2.H|_{B_{\reg}}^{n-1}).A|_{X_0}^{n-1}=0
\end{align*}
because $\dim B=n$. Hence $q_X(d,f^*h)=0$, and due to $q_X(f^*h,a)>0$ also ${q_X(d)=0}$, which gives $\tilde{D}^{2n}=0$.

The vanishing of $q_X(d,f^*h)$, $q_X(d)$, $q_X(f^*h)$ give the Fujiki relation $(\tilde{D}-\lambda f^*H)^{2n}=c_X\cdot q_X(d-\lambda f^*h)^n=0$ for every $\lambda\in\R$. For $\lambda\defeq\frac{\tilde{D}.A^{2n-1}}{f^*H.A^{2n-1}}\in\Q$ we get also $(\tilde{D}-\lambda f^*H).A^{2n-1}=0$, thus the $t^{2n-1}$ terms of the Fujiki relation
\[
(\tilde{D}-\lambda f^*H+tA)^{2n}=c_X\cdot(2tq_X(d-\lambda f^*h,a)+t^2q_X(a))^n
\]
vanish. So $q_X(d-\lambda f^*h,a)=0$ and therefore $d-\lambda f^*h\in a^{\perp}\cap\H^{1,1}(X)\cap\H^2(X,\,\R)$. The form $q_X$ is by Proposition~\ref{decomposition} negative definite on this space, so $d=\lambda f^*h$ in $\H^2(X,\,\Q)$ and $\tilde{D}\equiv\lambda f^*H$.

Now, if like in Matsushita's setting $D$ is $\Q$-Cartier, we get $D\equiv\lambda H$ by the projection formula, so $\rho(B)=1$. In the general case we show that $D$ is \emph{numerically $\Q$-Cartier} in the sense of \cite{HMP15}. This means that we have to construct for any resolution $\mu\colon\tilde{B}\ra B$ a $\mu$-trivial $\Q$-Cartier divisor $D'$ on $\tilde{B}$ with $\mu_*D'=D$.

For every resolution $\mu\colon\tilde{B}\ra B$ the fiber product $X\times_{B}\tilde{B}$ consists of the exceptional components of $\mu$ and a component $Y$ that is isomorphic to the graph of $f$ and lies birational over $X$. Hence we can construct a commutative diagram of resolutions
\begin{equation}
			\begin{gathered}
			\label{commonres}
			\begin{tikzpicture}[scale=1.5]
			\node (A) at (0,1) {$\tilde{X}$};
			\node (B) at (1,1) {$X$};
			\node (C) at (0,0) {$\tilde{B}$};
			\node (D) at (1,0) {$B$};
			
			\path[->,font=\scriptsize]
			(A) edge node[above]{$\nu$} (B)
			(C) edge node[above]{$\mu$} (D)
			(A) edge node[left]{$\tilde{f}$} (C)                        
			(B) edge node[right]{$f$} (D);
			\end{tikzpicture}
			\end{gathered}
\end{equation}
by taking $\tilde{X}$ to be a resolution of $Y$, together with the morphisms $\nu,\tilde{f}$ induced by the projections from the fiber product onto $X,B$. We consider $D'\defeq \tilde{f}_*\nu^*\tilde{D}$. As $\nu^*\tilde{D}\equiv\lambda\nu^*f^*H=\lambda\tilde{f}^*\mu^*H$, we have $D'\equiv\lambda\mu^*H$, which is $\mu$-trivial with ${\mu_*D'=\mu_*\tilde{f}_*\nu^*\tilde{D}=f_*\nu_*\nu^*\tilde{D}=D}$, so $D$ is numerically $\Q$-Cartier.

Fujino's result \cite[Theorem~1.2]{Fuj99} shows that $B$ is a klt variety, so it has rational singularities. Hence by \cite[Theorem~5.11]{HMP15} $D$ is $\Q$-Cartier, so $B$ is $\Q$-factorial.
\end{proof}

\subsection{Fibrations of primitive symplectic varieties}

We give an example showing that part~\ref{fano} of Theorem~\ref{fibration2} can fail for an irreducible symplectic variety. The idea is due to Matsushita and was worked out by Sawon \cite[p.~7f]{Mat01}, \cite[Lemma~15]{Saw14}.

\begin{ex}
\label{ex}
Consider the elliptic curve $E=\C/\Gamma$ with the lattice $\Gamma=\Z+\Z\cdot\zeta_6$, where $\zeta_6$ is a primitive sixth root of unity, and the $6$-torus $T\defeq E^6$. The maps
\begin{align*}
\C^6\ra&\;\C^6\\
(z_1,\ldots,z_6)\mapsto&\;(\zeta_6z_1,\zeta_6^5z_2,-z_3,-z_4,\zeta_6^2z_5,\zeta_6^4z_6)\\
(z_1,\ldots,z_6)\mapsto&\;(z_5,z_6,z_1,z_2,z_3,z_4)
\end{align*}
generate a group of automorphisms of $\C^6$ preserving $\Gamma^6\subset\C^6$, which induces a subgroup $G\subset\Aut T$. Then $X\defeq T/G$ is an irreducible symplectic variety with ${\omega\defeq dz_1\wedge dz_2+dz_3\wedge dz_4+dz_5\wedge dz_6}$, \cite[Proposition~2.4]{Bea00}. However, it is not primitive symplectic, as it carries holomorphic 3-forms. The group $G$ also acts on the subtorus $T'$ with coordinates $(z_1,z_3,z_5)$ and the quotient $B\defeq T'/G$ carries the holomorphic form $dz_1\wedge dz_3\wedge dz_5$, hence $K_B=0$, and $B$ has canonical singularities due to the Reid--Tai criterion~\cite[Theorem~3.21]{Kol13}.
\end{ex}

\begin{proof}[Proof of Theorem~\ref{fibration2}]
Let $(X,\omega)$ be a primitive symplectic variety with a morphism $f\colon X\ra B$ like in Theorem~\ref{fibration2}. We have to show the following two parts.

\begin{enumerate}
	\item\label{basefano} The base variety $B$ is Fano.
	\item\label{basepn} If $B$ is smooth, then $B\cong\P^n$.
\end{enumerate}

\subsubsection*{Proof of~\eqref{basefano}.}
As $B$ is $\Q$-factorial with $\rho(B)=1$ by Theorem~\ref{fibration}, the canonical divisor $K_B$ is $\Q$-Cartier with $K_B\equiv tH$ for a $t\in\Q$. By the Iitaka conjecture, which due to Kawamata is known to hold in this special case, see Theorem~\ref{iitaka} in the appendix, we have $\kappa(B)\leq0$, so $t\leq0$. Therefore we only need to exclude the case $t=0$. Up to here this is the same idea as in \cite[step 5]{Mat99}.

We assume that $K_B$ is numerically trivial. Applying \cite[Theorem~8.2]{Kawamata85} to a terminal model of $B$ shows that there is a ${d\in\Np}$, such that $\O_B(dK_B)\cong\O_B$ is trivial. As the canonical sheaf is a reflexive sheaf that is locally free on $X_{\reg}$, we can go over to a ramified cyclic covering like in \cite[2.52--2.53]{KM98}. This is a finite morphism $p\colon\hat{B}\ra B$ with ${p^*\O_B(K_B)=\O_{\hat{B}}}$ that is \'{e}tale of degree $d$ over $X_{\reg}$ and possibly branched over $X_{\sing}$. Hence $p$ is quasi-\'{e}tale because $X$ is normal. The variety $\hat{B}$ is by \cite[Proposition~5.20]{KM98} also a klt variety. We consider the fiber product $X\times_B\hat{B}$ and take the normalization $\hat{X}$ of any component that lies over $X_{\reg}$. The projections from the fiber product give a lift $\hat{f}\colon \hat{X}\ra\hat{B}$ of the morphism $f$ with $K_{\hat{B}}=0$ and a quasi-\'{e}tale morphism $\pi\colon\hat{X}\ra X$. There is a non-zero reflexive form $\alpha\in\H^0(\hat{B},\,\Omega^{[n]}_{\hat{B}})$ and, as $\hat{B}$ is a klt variety, we can pull it back by Theorem~\ref{extension} to a non-zero reflexive form ${\hat{f}^*\alpha\in\H^0(\hat{X},\,\Omega^{[n]}_{\hat{X}})}$. As $\hat{f}^*\alpha$ is not zero, $n$ has to be even and $\hat{f}^*\alpha$ a multiple of $\pi^*(\omega)^{n/2}$ because $X$ is primitive symplectic. This is a contradiction because $\hat{f}^*\alpha^2=0$, but $\pi^*\omega^n$ has no zeros.

\subsubsection*{Proof of~\eqref{basepn}.}
As $X$ is primitive symplectic, every terminal model $\pi\colon Y\ra X$ of $X$ is also primitive symplectic by Proposition~\ref{termisprimitive}.  Then ${\h^p(Y,\,\O_Y)=\h^0(Y,\,\Omega_X^{[p]})}$ for all $p$ by Hodge symmetry, \cite[Proposition~6.9]{GKP16}, so $Y$ is \emph{cohomologically irreducible symplectic} in Matsushita's terminology, \cite[Definition~1.6]{Mat15}. Applying Matsushita's result \cite[Theorem~1.10]{Mat15} to $f\circ\pi\colon Y\ra B$ shows $B\cong\P^n$ if $B$ is smooth.
\end{proof}

\section*{Appendix}

We explain here some probably well-known results used in the main proofs above. Lemma~\ref{smoothfiber} can be seen as a special case of \cite[Theorem~I.6.5]{Kol96}. Theorems~\ref{torsionfree} and~\ref{iitaka} are easy but useful consequences of results of Koll\'{a}r and Kawamata.

\begin{lem}[Smoothness at general fibers]
	\label{smoothfiber}
	Let $f\colon X\ra Y$ be a surjective morphism of complex varieties with a general fiber $F$. If $F$ is smooth at a point $x\in F$, then $X$ is smooth at $x$.
\end{lem}
\begin{proof}
	Let $X,\,Y$ be $n+r$ and $n$-dimensional, respectively. A general fiber $F$ lies over a smooth point $y\in Y$ and has the dimension $r$ by the fiber dimension theorem.
	Consider a smooth point $x\in F$ with maximal ideals ${\fm_{X,x}\subset\O_{X,x}}$ and ${\fm_{F,x}\subset\O_{F,x}}$ and the maximal ideal $\fm_{Y,y}\subset\O_{Y,y}$ of $y$. As $y\in Y$ and $x\in F$ are regular, we can write ${\fm_{F,x}=(\overline{\phi_1},\ldots,\overline{\phi_r})}$ and ${\fm_{Y,y}=(\psi_1,\ldots,\psi_n)}$. The ideal sheaf sequence for $F\subset X$ allows us to extend the $\overline{\phi_i}$ to functions ${\phi_i\in\fm_{X,x}\subset\O_{X,x}}$. Then it can easily be checked that we get a short exact sequence
	\[
	0\ra f^*\mathfrak{m}_{Y,y}\ra\fm_{X,x}\ra\fm_{F,x}\ra0
	\]
	that shows $\fm_{X,x}=(f^*\psi_1,\ldots,f^ *\psi_n,\phi_1,\ldots,\phi_r)$. Hence $x\in X$ is regular and also smooth because we work over a perfect field. 
\end{proof}

\begin{lem}[Singularities of general fibers]
	\label{singfiber}
	Let $f\colon X\ra Y$ be a surjective morphism of complex varieties with connected fibers, where $X$ is normal.  Then the general fiber $F$ is irreducible with ${F_{\sing}=F\cap X_{\sing}}$. When $X$ has canonical (resp.\,terminal) singularities, the general fiber also has canonical (resp.\,terminal) singularities.
\end{lem}
\begin{proof}
A general fiber $F$ lies over a smooth point $y\in Y$.
Let $L_y$ denote the linear system of hyperplanes passing through $y$.
Then the fiber $F$ is the base locus of the linear system $f^*L_y$.
The general member of $f^*L_y$ is normal with singularities contained in $X_{\sing}$ \cite[Theorem~1.7.1.1]{BS95}.
It follows by induction over the codimension of $F$ that $F$ is normal with $F_{\sing}\subset F\cap X_{\sing}$. 
On the other hand, a smooth point of $F$ is by Lemma~\ref{smoothfiber} also a smooth point of $X$.
This gives the equality $F_{\sing}=F\cap X_{\sing}$. As $F$ is normal and connected, it is also irreducible.

When $\nu\colon\tilde{X}\ra X$ is a resolution, the same argument shows that the strict transform $\tilde{F}$ of $F$ is smooth.
The adjunction formula allows to compare the minimal discrepancies as ${\mathrm{discr}(F)\geq \mathrm{discr}(X)}$, which shows that $F$ also has canonical (resp. terminal) singularities if $X$ does.

\qedhere
\end{proof}

\begin{thm}[Koll\'{a}r's torsion-freeness]
\label{torsionfree}
Let $f\colon X\ra B$ be a surjective morphism with connected fibers between complex projective varieties, where $X$ has rational singularities. Then all higher direct image sheaves $R^{i}f_*\omega_X$ for $i\geq0$ are torsion-free sheaves on $B$.
\end{thm}
\begin{proof}
Let $\nu\colon \tilde{X}\ra X$ be a resolution of singularities. We consider the Grothendieck spectral sequence for the functors $\nu_*$, $f_*$ and the canonical sheaf $\omega_{\tilde{X}}=\Omega^n_{\tilde{X}}$.
\[
E^{pq}_2=R^pf_*(R^q\nu_*\omega_{\tilde{X}})\Ra R^{p+q}(f\circ\nu)_*\omega_{\tilde{X}}
\]
By Grauert-Riemenschneider vanishing all higher direct images $R^q\nu_*\omega_{\tilde{X}}$ vanish for $q>0$, \cite[Theorem~4.3.9]{LazI}. Hence we are only left with the row $E^{p0}_2$ and the spectral sequence degenerates already on $E_2$. As $X$ has rational singularities, we have ${\nu_*\omega_{\tilde{X}}=\omega_X}$ by \cite[Lemma~5.12]{KM98}. This gives us
\[
R^pf_*\omega_X\cong R^pf_*(\nu_*\omega_{\tilde{X}})\cong R^p(f\circ\nu)_*\omega_{\tilde{X}}.
\]
The latter sheaf is torsion-free by Koll\'{a}r's result \cite[Theorem~2.1]{Kol86}.
\end{proof}

\begin{thm}[Reduction of the Iitaka conjecture to the minimal model program]
\label{iitaka}
Let $f\colon X\ra B$ be a surjective morphism with connected fibers between normal, complex projective varieties. If the general fiber $F$ has a \emph{good minimal model}, in other words $F$ is birational to a variety with semiample canonical divisor, then $\kappa(X)\geq\kappa(F)+\kappa(B)$.
\end{thm}
\begin{proof}
 We construct a commutative diagram like diagram~\ref{commonres} that lifts $f$ to a morphism $\tilde{f}$ between resolutions $\tilde{X},\tilde{B}$ of $X,B$. The fibers of $\tilde{f}$ are connected. The general fiber of $\tilde{f}$ is smooth by Lemma \ref{singfiber} and lies over the general fiber of $f$. As the Kodaira dimension is a birational invariant, Theorem~\ref{iitaka} follows now from the smooth case that was proved by Kawamata, \cite[Corollary 1.2]{Kawamata85}.
\end{proof}


\ifx\undefined\bysame
\newcommand{\bysame}{\leavevmode\hbox to3em{\hrulefill}\,}
\fi

\end{document}